\documentclass{article}

\usepackage{amsmath,amssymb,amsthm}

\theoremstyle{definition} \newtheorem{definition}{Definition}
\theoremstyle{definition} \newtheorem{hypothesis}{Hypothesis}
\theoremstyle{definition} \newtheorem*{remark}{Remark}
\theoremstyle{plain}    \newtheorem{theorem}{Theorem}
                          \newtheorem{corollary}[theorem]{Corollary}
                          \newtheorem{lemma}[theorem]{Lemma}

\title{A Maximal Inequality for $p$th Power of Stochastic Convolution Integrals}
\author{Erfan Salavati and Bijan Z. Zangeneh\\ \\
Department of Mathematical Sciences\\
Sharif University of Technology\\
Tehran, Iran}
\date{}

\begin{document}
\maketitle

\begin{abstract}
	An inequality for the $p$th power of the norm of a stochastic convolution integral in a Hilbert space is proved. The inequality is stronger than analogues inequalities in the literature in the sense that it is pathwise and not in expectation.
	
	An application of this inequality is provided for the semilinear stochastic evolution equations with L\'evy noise and monotone nonlinear drift. The existence and uniqueness of the mild solutions in $L^p$ for theses equations is proved and a sufficient condition for exponential asymptotic stability of the solutions is derived.
	
\end{abstract}

\vspace{2mm}
\vspace{2mm}


\section{Introduction}\label{section: introduction}

Stochastic convolution integrals appear in many fields of stochastic analysis. They are integrals of the form
\[ X_t=\int_0^t S_{t-s} dM_s \]
where $M_t$ is a martingale with values in a Hilbert space. Although they are generalization of stochastic integrals but they are different in many ways. For example they are not semimartingales in general and hence the usual results on semimartingales, such as maximal inequalities (i.e. inequalities for $\sup_{0\le s\le t} \|X_s\|$) and existence of c\`adl\`ag versions could not be applied directly to them.

Among first studies in this field one can note the works of Kotelenez~\cite{Kotelenez-1982} and Ichikawa~\cite{Ichikawa} where they consider stochastic convolution integrals with respect to general martingales. They prove a maximal inequality in $L^2$ for stochastic convolution integrals (Theorem~\ref{Theorem: Kotelenez inequality}).

Stochastic convolution integrals arise naturally in proving existence, uniqueness and regularity of the solutions of semilinear stochastic evolution equations,
\[ dX_t = A X_t dt + f(t,X_t) dt + g(t,X_t) dM_t \]
where $A$ is the generator of a $C_0$ semigroup of linear operators on a Hilbert space and $M_t$ is a martingale. The case that the coefficients are Lipschitz operators is studied well and the theorems of existence, uniqueness and continuity with respect to initial data for the solutions in $L^2$ is proved, see e.g~Kotelenez~\cite{Kotelenez-1984}. The proofs are based on the maximal inequality for stochastic convolution integrals, that is Theorem~\ref{Theorem: Kotelenez inequality}. 

These results have been generalized in several directions. One is the maximal inequality for $p$th power of the norm of stochastic convolution integrals. Tubaro has proved an upper estimate for $E[\sup_{0\le s\le t}\vert x(s)\vert ^p]$ with $p\ge2$ in the case that $M_t$ is a real Wiener process.  Ichikawa~\cite{Ichikawa} has proved maximal inequality for $p$th power, $p\ge 2$ in the special case that $M_t$ is a Hilbert space valued continuous martingale. The case of general martingale is proved by Zangeneh~\cite{Zangeneh-Paper} for $p\ge 2$ (see Theorem~\ref{theorem:Burkholder type inequality}). Hamedani and Zangeneh~\cite{Hamedani-Zangeneh-stopped} have generalized the maximal inequality to $0< p <\infty$.

%Brzezniak and Hausenblas~\cite{Brzezniak-Hausenblas} have proved a maximal inequality for the $p$th power of stochastic convolutions driven by Poisson random measures.

Brzezniak, Hausenblas and Zu~\cite{Brzezniak-Hausenblas-Zhu} have derived a maximal inequality for $p$th power of the norm of stochastic convolutions driven by Poisson random measures.

%Stochastic convolutions driven by martingales: Maximal inequalities and exponential integrability, E Hausenblas, J Seidler - Stochastic Analysis and Applications, 2007 - Taylor & Francis

%On stochastic convolution in Banach spaces and applications, Z Brzeźniak - … An International Journal of Probability and Stochastic …, 1997 - Taylor & Francis

As far as we know, the maximal inequalities proved for stochastic convolution integrals in the literature all involve expectations. The only result that provides a pathwise (almost sure) bound is Zangeneh~\cite{Zangeneh-paper} in which is proved Theorem~\ref{theorem:ito type inequality} called It\"o type inequality. This inequality provides a pathwise estimate for the square of the norm of stochastic convolution integrals and is the generalization of the It\"o formula to stochastic convolution integrals.

In Section~\ref{section: Stochastic Convolution Integrals} we define and state some results about stochastic convolution integrals that will be used in the sequel.

In Section~\ref{section: Ito type inequality for pth power} we state and prove the main result of this article, i.e. Theorem~\ref{theorem: Ito type inequality for pth power}, which provides a pathwise bound for the $p$th power of stochastic convolution integrals with respect to general martingales. The special case that the martingale is an It\"o integral with respect to a Wiener process has been proved by Jahanipour and Zangeneh~\cite{Jahanipour-Zangeneh}.

The pathwise nature of Theorem~\ref{theorem: Ito type inequality for pth power} enables one to apply it to semilinear stochastic evolution equations with non Lipschitz coefficients. We consider the drift term to be a monotone nonlinear operator and the noise term to be a compensated Poisson random measure and prove the exitence of the mild solution in $L^p$ in Theorem~\ref{theorem: existence in L^p}. The precise assumptions on coefficients wll be stated in Section~\ref{section: Semilinear Stochastic Evolution Equations with Levy Noise and Monotone Nonlinearity}. An auxiliary result is a Bichteler-Jacod inequality in Hilbert spaces proved in Theorem~\ref{theorem: maximal inequality Poisson}. This result has been stated and proved before in the literature, for example in~\cite{Marinelli-Prevot-Rockner}, but we give a new proof for it. We also show the exponential stability of the mild solutions under certain conditions in Theorem~\ref{theorem: stability}.

\section{Stochastic Convolution Integrals}\label{section: Stochastic Convolution Integrals}
Let $H$ be a separable Hilbert space with inner product $\langle \, , \, \rangle$. Let $S_t$ be a $C_0$ semigroup on $H$ with infinitesimal generator $A:D(A)\to H$. Furthermore we assume the exponential growth condition on $S_t$, i.e. there exists a constant $\alpha$ such that $\| S_t \| \le e^{\alpha t}$. If $\alpha=0$, $S_t$ is called a contraction semigroup.

In this section we review some properties and results about convolution integrals of type $X_t=\int_0^t S_{t-s} dM_s$ where $M_t$ is a martingale. These are called stochastic convolution integrals. Kotelenez~\cite{Kotelenez-1984} gives a maximal inequality for stochastic convolution integrals.

\begin{theorem}[Kotelenez,~\cite{Kotelenez-1984}] \label{Theorem: Kotelenez inequality}
    Assume $\alpha\ge 0$. There exists a constant $\mathbf{C}$ such that for any $H$-valued c\`adl\`ag locally square integrable martingale $M_t$ we have
    \[ \mathbb{E} \sup_{0\le t\le T} \|\int_0^t S_{t-s}dM_s\|^2 \le \mathbf{C} e^{4\alpha T} \mathbb{E}[M]_T.\]
\end{theorem}

\begin{remark}
    Hamedani and Zangeneh~\cite{Hamedani-Zangeneh-stopped} generalized this inequality to a stopped maximal inequality for $p$-th moment ($0<p<\infty$) of stochastic convolution integrals.
\end{remark}

Because of the presence of monotone nonlinearity in our equation, we need a pathwise bound for stochastic convolution integrals. For this reason the following pathwise inequality for the norm of stochastic convolution integrals has been proved in Zangeneh~\cite{Zangeneh-Paper}.

\begin{theorem}[It\"o type inequality, Zangeneh~\cite{Zangeneh-Paper}]\label{theorem:ito type inequality}
    Let $Z_t$ be an $H$-valued c\`adl\`ag locally square integrable semimartingale. If
    \[ X_t=S_t X_0 + \int_0^t S_{t-s}dZ_s, \]
    then
    \begin{equation*}
        \lVert X_t \rVert ^2 \le e^{2\alpha t}\lVert X_0 \rVert ^2 + 2 \int_0^t {e^{2\alpha (t-s)}\langle X_{s-} , d Z_s \rangle}+\int_0^t {e^{2\alpha (t-s)}d[Z]_s},
    \end{equation*}
    where $[Z]_t$ is the quadratic variation process of $Z_t$.
\end{theorem}

We state here the Burkholder-Davis-Gundy (BDG) inequality and a corollary to it, for future reference.

\begin{theorem}[Burkholder-Davis-Gundy (BDG) inequality] \label{theorem:BDG inequality}
    For every $p \ge 1$ there exists a constant $\mathcal{C}_p >0$ such that for any real valued square integrable cadlag martingale $M_t$ with $M_0=0$ and for any $T\ge0$,
        \[ \mathbb{E}\mathop{\sup}\limits_{0\le t\le T} |M_t|^p \le \mathcal{C}_p \mathbb{E} [M]_T^\frac{p}{2}. \]
\end{theorem}
\begin{proof}
    See \cite{Peszat-Zabczyk}, page 37, and the reference there.
\end{proof}

\begin{corollary} \label{corollary: BDG}
    Let $p\ge 1$ and $\mathcal{C}_p$ be the constant in the BDG inequality and $M_t$ be an $H$-valued square integrable cadlag martingale and $X_t$ an $H$-valued adapted process and $T\ge 0$. Then for $K>0$,
    \begin{eqnarray*}
        \mathbb{E}\mathop{\sup}\limits_{0\le t\le T} \left|\int_0^t \langle X_s,dM_s\rangle\right|^p &\le& \mathcal{C}_p\mathbb{E}\left((X_t^*)^p[M]_t^\frac{p}{2}\right)\\
        &\le& \frac{\mathcal{C}_p}{2K} \mathbb{E}(X_t^*)^{2p} + \frac{\mathcal{C}_p K}{2} \mathbb{E}[M]_t^p.
    \end{eqnarray*}
    where $X_t^*=\mathbb{E}\mathop{\sup}\limits_{0\le t\le T} \|X_t\|$.
\end{corollary}
\begin{proof}
    See~\cite{Zangeneh-Paper}, Lemma 4, page 147.
\end{proof}

We will need also the following inequality which is an analogous of Burkholder-Davies-Gundy inequality for stochastic convolution integrals.

\begin{theorem}[Burkholder Type Inequality, Zangeneh~\cite{Zangeneh-Paper}, Theorem 2, page 147]\label{theorem:Burkholder type inequality}
    Let $p\ge 2$ and $T>0$. Let $S_t$ be a contraction semigroup on $H$ and $M_t$ be an $H$-valued square integrable c\`adl\`ag martingale for $t\in[0,T]$. Then
    \begin{equation*}
        \mathbb{E}\sup\limits_{0 \le t \le T}\|{\int_0^t S_{t-s}dM_s}\|^p \le K_p \mathbb{E}([M]_T^\frac{p}{2})
    \end{equation*}
    where $K_p$ is a constant depending only on $p$.
\end{theorem}

\section{It\"o Type Inequality for $p$th Power}\label{section: Ito type inequality for pth power}

We use the notion of semimartingale and It\"o's formula as is described in Metivier~\cite{Metivier}.

\begin{theorem}[It\"o type Inequality for $p$th power]\label{theorem: Ito type inequality for pth power}
    Let $p\ge 2$. Assume $Z(t)=V(t) + M(t)$ is a semimartingale where $V(t)$ is an $H$-valued process with finite variation $|V|(t)$ and $M(t)$ is an $H$-valued square integrable martingale with quadratic variation $[M](t)$. Assume that
        \[ \mathbb{E} [M](T)^{\frac{p}{2}} <\infty \qquad \qquad \mathbb{E}|V|(T)^p <\infty \]
    Let $X_0(\omega)$ be $\mathcal{F}_0$ measurable and square integrable. Define $X(t)=S(t) X_0 + \int_0^t S(t-s) dZ(s)$. Then we have
    \begin{eqnarray*}
        \|X(t)\|^p &\le& e^{p\alpha t}\|X_0\|^p + p\int_0^t e^{p\alpha (t-s)} \|X(s^-)\|^{p-2} \langle X(s^-) , dZ(s) \rangle \\
        && + \frac{1}{2}p(p-1)\int_0^t e^{p\alpha (t-s)} \|X(s^-)\|^{p-2} d[M]^c(s)\\
        && + \sum_{0\le s\le t} e^{p\alpha (t-s)} \left( \|X(s)\|^p -\|X(s^-)\|^p - p \|X(s^-)\|^{p-2} \langle X(s^-) , \Delta X(s) \rangle \right)
    \end{eqnarray*}

%%%%%%%%%%%% Case alpha=0 %%%%%%%%%%%%%
%        \|X(t)\|^p &\le& \|X_0\|^p + p\int_0^t \|X(s^-)\|^{p-2} \langle X(s^-) , dZ(s) \rangle + \frac{1}{2}p(p-1)\int_0^t \|X(s^-)\|^{p-2} d[M]^c(s)\\
 %       && + \sum_{0\le s\le t} \left( \|X(s)\|^p -\|X(s^-)\|^p - p \|X(s^-)\|^{p-2} \langle X(s^-) , \Delta X(s) \rangle \right)
%%%%%%%%%%%%%%%%%%%%%%%%%%%%%%%%%%%%%%%
%    and????????????????????
 %       \begin{eqnarray*}
  %      \|X_t\|^p &\le& \|X_0\|^p + p\int_0^t \|X_{s-}\|^{p-2} \langle X_{s-} , dZ_{s-} \rangle + \frac{1}{2}p(p-1)\int_0^t \|X_{s-}\|^{p-2} d[M]_s\\
   %     && + \sum_{0\le s\le t} \left( \|X_s\|^p -\|X_{s-}\|^p - p \|X_{s-}\|^{p-2} \langle X_{s-} , \Delta X_s \rangle -\frac{1}{2}p(p-1)\|X_{s-}\|^{p-2} \|\Delta X_s\|^2 \right)
    %\end{eqnarray*}????????????????
\end{theorem}

\begin{remark}
    \begin{enumerate}
        \item For $p=2$ the theorem implies the Ito type inequality(Theorem~\ref{theorem:ito type inequality}).
        \item If $M$ is a continuous martingale then the inequality takes the simpler form
        \begin{eqnarray*}
            \|X(t)\|^p &\le& e^{p\alpha t}\|X_0\|^p + p\int_0^t e^{p\alpha (t-s)} \|X(s^-)\|^{p-2} \langle X(s^-) , dZ(s) \rangle \\
            && + \frac{1}{2}p(p-1)\int_0^t e^{p\alpha (t-s)} \|X(s^-)\|^{p-2} d[M](s)
        \end{eqnarray*}
    \end{enumerate}
\end{remark}

Before proceeding to the proof of theorem we state and prove some lemmas.

\begin{lemma}\label{lemma: alpha=0_Ito type inequality}
        It suffices to prove theorem~\ref{theorem:ito type inequality} for the case that $\alpha=0$.
    \end{lemma}
\begin{proof} Define
        \begin{gather*}
           \tilde{S}(t)= e^{-\alpha t} S(t) ,\qquad \tilde{X}(t)=e^{-\alpha t}X(t) ,\\
           d\tilde{Z}(t)=e^{-\alpha t}dZ(t)
        \end{gather*}
        Now we have $d\tilde{X}(t)=\tilde{S}(t) X_0+\int_0^t \tilde(S)(t-s)d\tilde{Z}(s)$.
        Note that $\tilde{S}_t$ is a contraction semigroup. It is easy to see that the statement for $\tilde{X}_t$ implies the statement for $X(t)$.
\end{proof}

Hence from now on we assume $\alpha=0$.

\begin{lemma}[Ordinary It\"o's formula for $p$th power]\label{lemma: Ito's formula for pth power}
    Let $p\ge 2$ and assume that $Z(t)$ is an $H$-valued semimartingale. Then
    \[ \begin{array}{l} \|Z(t)\|^p \le \|Z(0)\|^p + p \int_0^t \|Z(s^-)\|^{p-2} \langle Z(s^-) , dZ(s) \rangle + \frac{p(p-1)}{2} \int_0^t \|Z(s^-)\|^{p-2} d[M]^c(s)\\ + \sum_{0\le s\le t} \left( \|Z(s)\|^p - \|Z(s^-)\|^p - p \|Z(s^-)\|^{p-2} \langle Z(s^-) , \Delta Z(s) \rangle \right) \end{array} \]
\end{lemma}
\begin{proof}
    Use It\"o's formula (Metivier~\cite{Metivier}, Theorem 27.2, Page 190) for $\varphi(x)=\|x\|^p$ and note that
    \[\varphi'(x)(h)=p\|x\|^{p-2}\langle x , h\rangle,\]
    \[ \varphi''(x)(h\otimes h)=\frac{1}{2}p(p-2) \|x\|^{p-4}\langle x , h\rangle \langle x , h \rangle + \frac{1}{2}p\|x\|^{p-2}\langle h , h \rangle \le \frac{1}{2}p(p-1) \|x\|^{p-2} \|h\|^2\]
\end{proof}

\begin{lemma}\label{lemma: D valued deterministic mild solution}
    Assume $v:[0,T]\to D(A)$ is a function with finite variation (with respect to the norm of $D(A)$) denoted by $|v|(t)$. Assume that $u_0 \in D(A)$. Let $u(t)=S(t) u_0 + \int_0^t S(t-s) dv(s)$. Then $u(t)$ is $D(A)$-valued and satisfies
        \[  u(t)=u_0+\int_0^t A u(s) ds + v(t)\]
\end{lemma}
\begin{proof}
    (see also Curtain and Pritchard page 30 Theorem 2.22 for the special case $dv(t)=f(t)dt$.)
    Let $q(t)$ be the Radon-Nikodym derivative of $v(t)$ with respect to $|v|(t)$, i.e, $q(t)$ is a $D(A)$-valued function which is Bochner measurable with respect to $d|v|(t)$ and $v(t)=\int_0^t q(s) d|v|(s)$. We know that for every $t\in [0,T]$, $\|q(t)\|\le 1$.

	Recall from semigrop theory that one can equip $D(A)$ with an inner product by defining $\langle x , y \rangle_{D(A)} := \langle x , y \rangle + \langle Ax , Ay \rangle$. By closedness of $A$ it follows that under this inner product $D(A)$ is a Hilbert space and $A:D(A)\to H$ is a bounded linear map. Note that $S(t)$ is also a semigroup on $D(A)$. Hence $u(t)$ is a convolution integral in $D(A)$ and hence has it's value in $D(A)$. We use the following two simple identities that hold in $D(A)$:
        \[  S(t)x = x + \int_0^t A S(r) x dr,\qquad S(t-s)x = x +  \int_s^t S(r-s)Ax dr \]
    We have
        \[ \begin{array}{l l}
            u(t) & = S(t) u_0 + \int_0^t S(t-s) dv(s)\\
                 & = S(t) u_0 + \int_0^t S(t-s) q(s) d|v|(s)\\
                 & = u_0 + \int_0^t A S(r) u_0 dr + \int_0^t \left( q(s) + A \int_s^t S(r-s)q(s) dr\right) d|v|(s)\\
            \end{array} \]
    Now using Fubini's theorem we find
        \[ \begin{array}{l l}
                 & = u_0 + \int_0^t q(s) d|v|(s) + \int_0^t A S(r) u_0 dr + \int_0^t A \int_0^r S(r-s)q(s)d|v|(s) dr\\
                 & = u_0 + v(t) + \int_0^t A \left( S(r) u_0 dr + \int_0^r S(r-s)dv(s) dr\right)\\
                 & = u_0 + v(t) + \int_0^t A u(r) dr.
            \end{array} \]
\end{proof}

\begin{lemma}\label{lemma: D valued mild solution}
    Assume $V(t)$ is a $D(A)$-valued process with finite variation in $D(A)$ and $M(t)$ is a $D(A)$-valued square integrable martingale and $V(0)=M(0)=0$. Let $Z(t)=V(t) + M(t)$ and let $X_0$ be $D(A)$-valued and $\mathcal{F}_0$-measurable and define $X(t)=S(t) X_0 + \int_0^t S(t-s) dZ(s)$. Then $X(t)$ is $D(A)$-valued and satisfies the following stochastic integral equation in $H$:
        \[  X(t)= S(t) X_0 + \int_0^t A X(s) ds + Z(t)\]
\end{lemma}
\begin{proof}[Proof of Lemma]
    Note that $S(t)$ is also a semigroup on $D(A)$. Hence $X(t)$ is a stochastic convolution integral in $D(A)$ and hence has it's value in $D(A)$. Write $\overline{Y}(t)=S(t)X_0+\int_0^t S(t-s)dV(s)$ and $Y(t)=\int_0^t S(t-s) dM(s)$. Hence $X(t)=\overline{Y}(t)+Y(t)$. We can apply lemma~\ref{lemma: D valued deterministic mild solution} to term $\overline{Y}(t)$ and deduce $\overline{Y}(t)=X_0+\int_0^t A \overline{Y}(s) ds + V(t)$. Hence it suffices to prove $Y(t)=\int_0^t A Y(s) ds + M(t)$.

    Let $\{e_1,e_2,e_3,\ldots\}$ be a basis for Hilbert space $D(A)$. Define $\overline{M}^j(t)=\langle M(t),e_j \rangle$ and $M^k(t)=\sum_{j=1}^k \overline{M}^j(t)$. Let $Y^k(t)=\int_0^t S(t-s) dM^k(s)$. We use the following two simple identities that hold in $D(A)$:
        \[  S(t)x = x + \int_0^t A S(r) x dr,\qquad S(t-s)x = x + A \int_s^t S(r-s)x dr \]
    We have
        \[ \begin{array}{l l}
            Y^k(t)  & = \int_0^t S(t-s) dM^k(s)\\
                    & = \sum_1^k \int_0^t S(t-s) e_j d\overline{M}^j(s)\\
                    & = \sum_1^k \int_0^t\left( e_j + \int_s^t S(r-s)A e_j dr \right) d\overline{M}^j(s)\\
                    & = M^k(t) + \int_0^t \int_s^t S(r-s)A e_j dr d\overline{M}^j(s)\\
            \end{array} \]
    Now using stochastic Fubini theorem (see~\cite{Peszat-Zabczyk} Theorem 8.14 page 119) we find
        \[ \begin{array}{l l}
                    & = M^k(t) + \int_0^t \int_0^r S(r-s)A e_j d\overline{M}^j(s) dr\\
                    & = M^k(t) + \int_0^t A \left( \int_0^r S(r-s)d\overline{M}^j(s)\right)dr\\
                    & = M^k(t) + \int_0^t A Y^k(s) ds.
            \end{array} \]
    Hence we find
    \begin{equation}\label{equation: proof of lemma: D valued mild solution}
        Y^k(t)=M^k(t) + \int_0^t A Y^k(s) ds
    \end{equation}
    We have $\mathbb{E}\|M(T)-M^k(T)\|_{D(A)}^2\to 0$ and by Theorem~\ref{Theorem: Kotelenez inequality}
        \[ \mathbb{E} \sup_{0\le t\le T} \|Y(t)-Y^k(t)\|_{D(A)}^2 \le \mathbf{C} \mathbb{E}\|M(T)-M^k(T)\|_{D(A)}^2 \to 0\]
    and since $A:D(A)\to H$ is continuous $\mathbb{E} \sup_{0\le t\le T} \|AY(t)-AY^k(t)\|_H^2\to 0$ and hence $\mathbb{E}\|\int_0^t A Y(s) ds - \int_0^t A Y^k(s) ds\| \to 0$. Hence by taking limits from both sides of~\eqref{equation: proof of lemma: D valued mild solution} we get
        \[ Y(t)=M(t) + \int_0^t A Y(s) ds. \]
\end{proof}

\begin{proof}[Proof of Theorem~\ref{theorem: Ito type inequality for pth power}]
    By using Lemma~\ref{lemma: alpha=0_Ito type inequality}, we need only to prove for the case $\alpha=0$. In this case we have to prove
    \begin{equation}\begin{array}{l l}\label{equation: proof of Ito type inequality pth power 1}
        \|X(t)\|^p \le& \|X_0\|^p + p\int_0^t \|X(s^-)\|^{p-2} \langle X(s^-) , dZ(s) \rangle + \frac{1}{2}p(p-1)\int_0^t \|X(s^-)\|^{p-2} d[M]^c(s)\\
        & + \sum_{0\le s\le t} \left( \|X(s)\|^p -\|X(s^-)\|^p - p \|X(s^-)\|^{p-2} \langle X(s^-) , \Delta X(s) \rangle \right).
    \end{array} \end{equation}
    
    The main idea is that we approximate $M(t)$ and $V(t)$ by some $D(A)$ valued processes, and for $D(A)$ valued processes we use ordinary It\"o's formula. This is done by Yosida approximations. We recall some facts from semigroup theory in the following lemma. For proofs see Pazy~\cite{Pazy}.
    
\begin{lemma}\label{lemma: properties of yosida approximations}
    For $\lambda > 0$, $\lambda I - A$ is invertible. Let $R(\lambda)=\lambda(\lambda I - A)^{-1}$ and $A(\lambda)=A R(\lambda)$. We have:
        \begin{description}
          \item[(a)] $R(\lambda): H \to D(A)$ and $A(\lambda):H \to H$ are bounded linear maps.
          \item[(b)] for every $x\in H$, $\|R(\lambda)x\|_H \le \|x\|_H$ and $\langle x , A(\lambda) x \rangle \le 0$.
          \item[(c)] $R(\lambda)S(t)=S(t)R(\lambda)$ and for $x\in D(A)$, $R(\lambda) A x = A R(\lambda) x$.
          \item[(d)] for every $x\in H$, $\lim_{\lambda\to \infty} R(\lambda) x = x$ in $H$.
          \item[(e)] for every $x\in D(A)$, $\lim_{\lambda\to \infty} A(\lambda) x = Ax$.
        \end{description}
\end{lemma}

    Now for $n=1,2,3,\ldots$ Define:
    \[  V^n(t)=R(n)V(t),\qquad M^n(t)=R(n)M(t),\qquad Z^n(t)=V^n(t)+M^n(t)=R(n)Z(t)\]
    \[  X^n_0= R(n) X_0,\qquad X^n(t)=S(t) X^n_0 + \int_0^t S(t-s) dZ^n(s) \]
    According to Lemma~\ref{lemma: properties of yosida approximations}, $V^n(t)$ is a $D(A)$-valued finite variation process, $M^n(t)$ is a $D(A)$-valued martingale and $Z^n(t)$ is a $D(A)$-valued semimartingale. Hence by lemma~\ref{lemma: D valued mild solution}, $X^n(t)$ is an ordinary stochastic integral and hence we can apply Lemma~\ref{lemma: Ito's formula for pth power} to it and find
    \[ \begin{array}{l}
        \|X^n(t)\|^p \le \|X^n_0\|^p + p \int_0^t \|X^n(s^-)\|^{p-2} \langle X^n(s^-) , A X^n(s) ds + dV^n(s) + dM^n(s) \rangle\\
        + \frac{p(p-1)}{2} \int_0^t \|X^n(s^-)\|^{p-2} d[M^n]^c(s)+F^n
    \end{array} \]
    where
    \[ F^n = \sum_{0\le s\le t} \left( \|X^n(s)\|^p - X^n(s^-)^p - p \|X^n(s^-)\|^{p-2} \langle X^n(s^-) , \Delta Z^n(s) \rangle \right).\]
    Since $A$ is the generator of a contraction semigroup, we have $\langle A x , x \rangle \le 0$, hence we find
    \begin{equation}\label{equation: proof of Ito type inequality pth power 3} \begin{array}{l l}
        \underbrace{\|X^n(t)\|^p}_\mathbf{A^n} \le & \underbrace{\|X^n_0\|^p}_\mathbf{B^n} + p \underbrace{\int_0^t \|X^n(s^-)\|^{p-2} \langle X^n(s^-) , dV^n(s)\rangle}_\mathbf{C^n}\\
        & + p \underbrace{\int_0^t \|X^n(s^-)\|^{p-2} \langle X^n(s^-) , dM^n(s) \rangle}_\mathbf{D^n}\\
        & + \frac{p(p-1)}{2} \underbrace{\int_0^t \|X^n(s^-)\|^{p-2} d[M^n]^c(s)}_\mathbf{E^n}+F^n.
    \end{array} \end{equation}
    We claim that the inequality~\eqref{equation: proof of Ito type inequality pth power 3} (after choosing a suitable subsequence) converges term by term in to the following inequality and hence the following will be proved:
    \begin{equation*} \begin{array}{l l}
        \underbrace{\|X(t)\|^p}_\mathbf{A} \le & \underbrace{\|X_0\|^p}_\mathbf{B} + p \underbrace{\int_0^t \|X(s^-)\|^{p-2} \langle X(s^-) , dV(s)\rangle}_\mathbf{C}\\
        & + p \underbrace{\int_0^t \|X(s^-)\|^{p-2} \langle X(s^-) , dM(s) \rangle}_\mathbf{D}\\
        & + \frac{p(p-1)}{2} \underbrace{\int_0^t \|X(s^-)\|^{p-2} d[M]^c(s)}_\mathbf{E}+F
    \end{array} \end{equation*}
    where
    \[ F = \sum_{0\le s\le t} \left( \|X(s)\|^p - X(s^-)^p - p \|X(s^-)\|^{p-2} \langle X(s^-) , \Delta Z(s) \rangle \right).\]
    We prove this claim in several steps.

    \begin{description}
      \item[(Step 1)] We claim that $\mathbb{E} |V^n-V|(t)^p \to 0$.
        Let $q(t)$ be the Radon-Nykodim derivative of $V(t)$ with respect to $|V|(t)$. We know that for every $t$, $\|q(t)\|\le 1$. We have
            \[ \mathbb{E} |V^n-V|(t)^p = \mathbb{E}\left(\int_0^t \|(R(n)-I) q(s)\| d|V|(s)\right)^p \]
        Note that for every $s$ and $\omega$, $\|(R(n)-I) q(s)\| \le 2 $ and tends to zero and since $|V|(t)<\infty,\quad a.s.$ by the Lebesgue's dominated convergence theorem, $\int_0^t \|(R(n)-I) q(s)\| d|V|(s)\to 0,\quad a.s.$ and is dominated by $2|V|(t)$. Now since $\mathbb{E} |V|(t)^p < \infty$ and using the Lebesgue's dominated convergence theorem we find that $\mathbb{E}\left(\int_0^t \|(R(n)-I) q(s)\| d|V|(s)\right)^p\to 0$ and the claim is proved.

      \item[(Step 2)] We claim that $\mathbb{E} [M^n-M](t)^\frac{p}{2} \to 0$.
        Note that $[M^n-M](t)\le 2[M^n](t)+2[M](t)\le 4[M](t)$ and hence $[M^n-M](t)^{\frac{p}{2}}$ is dominated by
        $4^{\frac{p}{2}}[M](t)^{\frac{p}{2}}$. On the other hand $\mathbb{E}[M^n-M](t) = \mathbb{E}\|M^n(t)-M(t)\|^2\to 0$. Hence $[M^n-M](t)$ and consequently $[M^n-M](t)^{\frac{p}{2}}$ tend to 0 in probability and therefore by Lebesgue's dominated convergence theorem it's expectation also tends to 0.
      \item[(Step 3)]
        We claim that
        \begin{equation}\label{equation: proof of Ito type inequality pth power 2}
            \mathbb{E}\sup_{0\le s \le t} \| X^n(s)-X(s)\|^p \to 0.
        \end{equation}
        We have
        \[ \begin{array} { l l}
            \| X^n(s)-X(s)\|^p & \le 3^p \underbrace{\|S(s) ( X^n_0-X_0)\|^p}_\mathbf{A_1}\\
            & + 3^p \underbrace{\|\int_0^s S(s-r) d(V^n(r)-V(r))\|^p}_\mathbf{A_2}\\
            & + 3^p \underbrace{\|\int_0^s S(s-r) d(M^n(r)-M(r))\|^p}_\mathbf{A_3}.
        \end{array}\]
        For $\mathbf{A_1}$ we have
            \[ \mathbb{E}\sup_{0\le s \le t} \mathbf{A_1} \le \mathbb{E}\|X_0^n-X_0\|^p \to 0.\]
        For $\mathbf{A_2}$ we have
            \[ \mathbb{E}\sup_{0\le s \le t} \mathbf{A_2} \le  \mathbb{E} |V^n-V|(t)^p \to 0,\]
        where we have used Step 1.
        For $\mathbf{A_3}$, we use Burkholder type inequality (Theorem~\ref{theorem:Burkholder type inequality}) for $\alpha=0$ and find
        \[ \mathbb{E}\sup_{0\le s \le t} \mathbf{A_3} \le K_p \mathbb{E}\left( [M^n-M](t)^{\frac{p}{2}}\right)\to 0,\]
        where we have used Step 2. Hence~\eqref{equation: proof of Ito type inequality pth power 2} is proved.

      \item[(Step 4)] We claim that 
              \begin{equation}\label{equation: proof of Ito type inequality *}
            \mathbb{E}\sup_{0\le s \le t} \| X^n(s)\|^p \to \mathbb{E}\sup_{0\le s \le t} \| X(s)\|^p
        \end{equation}
      By triangle inequality,
        \[ \begin{array}{l} \left|\left(\mathbb{E}\sup_{0\le s \le t} \| X^n(s)\|^p\right)^{\frac{1}{p}} - \left(\mathbb{E}\sup_{0\le s \le t} \| X(s)\|^p\right)^{\frac{1}{p}}\right| \le \\
            \left( \mathbb{E} \left| \sup_{0\le s \le t} \|X^n(s)\| - \sup_{0\le s \le t} \|X(s)\| \right|^p \right)^{\frac{1}{p}} \le \\
            \left( \mathbb{E} \sup_{0\le s \le t} |\|X^n(s)\|-\|X(s)\||^p \right)^{\frac{1}{p}} \le \\
            \left( \mathbb{E} \sup_{0\le s \le t} \|X^n(s)-X(s)\|^p \right)^{\frac{1}{p}} \to 0 \end{array}\]
	where in the last line we have used Step 3. Hence~\eqref{equation: proof of Ito type inequality *} is proved and in particular the sequence $\mathbb{E}\sup_{0\le s \le t} \| X^n(s)\|^p$ is bounded for each $t$.

      \item[(Step 5)] We claim that $\mathbb{E}|\mathbf{C^n}-\mathbf{C}| \to 0$. We have
        \[\begin{array} {l}
            \mathbb{E}|\mathbf{C^n}-\mathbf{C}| \le \underbrace{\mathbb{E}|\int_0^t (\|X^n(s^-)\|^{p-2}-\|X(s^-)\|^{p-2})\langle X^n(s^-) , dV^n(s)\rangle |}_\mathbf{C^n_1}\\
            + \underbrace{\mathbb{E}|\int_0^t \|X(s^-)\|^{p-2}\langle X^n(s^-)-X(s^-) , dV^n(s)\rangle |}_\mathbf{C^n_2} +\\
            \underbrace{\mathbb{E}|\int_0^t \|X(s^-)\|^{p-2}\langle X(s^-) , d(V^n(s)-V(s))\rangle|}_\mathbf{C^n_3}.
        \end{array} \]

        For the term $\mathbf{C^n_1}$ we have,
        \begin{equation*}
        		\begin{array} {l} \mathbf{C^n_1} \le \mathbb{E} \left( (\sup |\|X^n(s^-)\|^{p-2} - \|X(s^-)\|^{p-2}|)(\sup \|X^n(s^-)\|) |V^n|(t)\right) \end{array}
        \end{equation*}

        Now using the simple inequality $ |a-b|^r \le |a^r - b^r|$ for $r\ge 1$ and $a,b\in\mathbb{R}^+$ we have $|\|X^n(s^-)\|^{p-2} - \|X(s^-)\|^{p-2}| \le |\|X^n(s^-)\|^p-\|X(s^-)\|^p|^{\frac{p-2}{p}}$. Substituting and using the Holder inequality we find
        \[ \le \left( \mathbb{E} \sup |\|X^n(s^-)\|^p -\|X(s^-)\|^p|\right)^\frac{p-2}{p} \left( \mathbb{E} \sup \|X^n(s^-)\|^p\right)^\frac{1}{p} \left( \mathbb{E}|V^n|(t)^p\right)^\frac{1}{p} \]
        The second term above is bounded (according to step 4) and the third term is bounded by $\left(\mathbb{E}|V|(t)^p\right)^\frac{1}{p}$ since $|V^n|(t)\le|V|(t)$. We claim that the first term, after choosing a subsequence, converges to zero.
        We know from Step 3 that $\mathbb{E}\sup_{0\le s \le t} \| X^n(s)-X(s)\|^p \to 0$. Hence we can choose a subsequence $n_k$ for which $\sup_{0\le s \le t} \| X^{n_k}(s)-X(s)\|^p \to 0, a.s$. We have also $\sup_{0\le s \le t} \|X(s)\|<\infty, a.s$, hence
        \[ \sup |\|X^{n_k}(s^-)\|^p -\|X(s^-)\|^p|\to 0, \quad a.s \]
On the other hand
        \begin{multline*}
\sup_{0\le s\le t} |\|X^{n_k}(s^-)\|^p -\|X(s^-)\|^p| \le \\
2^p \sup_{0\le s\le t} \|X^{n_k}(s^-)-X(s^-)\|^p + (2^p+1) \sup_{0\le s\le t} \|X(s^-)\|^p.        
        \end{multline*}

        Hence by dominated convergence theorem we have
        \[ \mathbb{E} \sup_{0\le s\le t} |\|X^{n_k}(s^-)\|^p -\|X(s^-)\|^p| \to 0 \]
        and therefore for the same subsequence $\mathbf{C^n_1}\to 0$.

        For the term $\mathbf{C^n_2}$ we have,
        \[ \mathbf{C^n_2} \le \mathbb{E}\left( (\sup_{0\le s\le t}\|X(s^-)\|^{p-2})(\sup_{0\le s\le t}\|X^n(s^-)-X(s^-)\|)|V^n|(t)\right)\]
        By Holder inequality we have
        \[ \le \left( \mathbb{E} \sup_{0\le s\le t}\|X(s^-)\|^p\right)^\frac{p-2}{p} \left( \mathbb{E} \sup_{0\le s\le t}\|X^n(s^-)-X(s^-)\|^p\right)^\frac{1}{p}\left(\mathbb{E} |V^n|(t)^p\right)^\frac{1}{p}. \]
        The first and third terms are bounded and the second term tends to zero by Step 3. Hence $\mathbf{C^n_2}\to 0$.

        For the term $\mathbf{C^n_3}$ we have,
            \[ \mathbf{C^n_3} \le \mathbb{E}\left( (\sup_{0\le s\le t}\|X(s^-)\|^{p-1}) |V^n-V|(t)\right) \]
        By Holder inequality we have
            \[ \le \mathbb{E} \left( \sup_{0\le s\le t}\|X(s^-)\|^p \right)^{\frac{p-1}{p}} \left(\mathbb{E} (|V^n-V|(t)^p)\right)^{\frac{1}{p}} \]
        where tends to 0 by Step 1. Hence $\mathbf{C^n_3}\to 0$.

      \item[(Step 6)] We claim that $\mathbb{E}|\mathbf{D^n}-\mathbf{D}| \to 0$. We have
        \[\begin{array} {l}
            \mathbb{E}|\mathbf{D^n}-\mathbf{D}| \le \underbrace{\mathbb{E}|\int_0^t (\|X^n(s^-)\|^{p-2}-\|X(s^-)\|^{p-2})\langle X^n(s^-) , dM^n(s)\rangle |}_\mathbf{D^n_1}\\
            + \underbrace{\mathbb{E}|\int_0^t \|X(s^-)\|^{p-2}\langle X^n(s^-)-X(s^-) , dM^n(s)\rangle |}_\mathbf{D^n_2} +\\
            \underbrace{\mathbb{E}|\int_0^t \|X(s^-)\|^{p-2}\langle X(s^-) , d(M^n(s)-M(s))\rangle|}_\mathbf{D^n_3}.
        \end{array} \]

        For the term $\mathbf{D^n_1}$ we use Corollary~\ref{corollary: BDG} for $p=1$ and find
            \[ \mathbf{D^n_1} \le \mathcal{C}_p \mathbb{E} \left( (\sup |\|X^n(s^-)\|^{p-2} - \|X(s^-)\|^{p-2}|)(\sup \|X^n(s^-)\|) [M^n](t)^{\frac{1}{2}}\right) \]
        Now using the simple inequality $ |a-b|^r \le |a^r - b^r|$ for $r\ge 1$ and $a,b\in\mathbb{R}^+$ we have $|\|X^n(s^-)\|^{p-2} - \|X(s^-)\|^{p-2}| \le |\|X^n(s^-)\|^p-\|X(s^-)\|^p|^{\frac{p-2}{p}}$. Substituting and using the Holder inequality we find
            \[ \le \mathcal{C}_p \left( \mathbb{E} \sup |\|X^n(s^-)\|^p -\|X(s^-)\|^p|\right)^\frac{p-2}{p} \left( \mathbb{E} \sup \|X^n(s^-)\|^p\right)^\frac{1}{p} \left( \mathbb{E}[M^n](t)^\frac{p}{2}\right)^\frac{1}{p} \]
        The second term above is bounded (according to step 4) and the third term is bounded by $\left( \mathbb{E}[M](t)^\frac{p}{2}\right)^\frac{1}{p}$ since $[M^n](t)\le[M](t)$. The first term, by the same arguments as in Step 5, after choosing a subsequence, converges to zero.

        For the term $\mathbf{D^n_2}$ we use Corollary~\ref{corollary: BDG} for $p=1$ and find
        \[ \mathbf{D^n_2} \le \mathcal{C}_p \mathbb{E}\left( (\sup_{0\le s\le t}\|X(s^-)\|^{p-2})(\sup_{0\le s\le t}\|X^n(s^-)-X(s^-)\|)[M^n](t)^\frac{1}{2}\right)\]
        By Holder inequality we have
        \[ \le \mathcal{C}_p \left( \mathbb{E} \sup_{0\le s\le t}\|X(s^-)\|^p\right)^\frac{p-2}{p} \left( \mathbb{E} \sup_{0\le s\le t}\|X^n(s^-)-X(s^-)\|^p\right)^\frac{1}{p}\left(\mathbb{E} [M^n](t)^\frac{p}{2}\right)^\frac{1}{p}. \]
        The first and third terms are bounded and the second term tends to zero by Step 3. Hence $\mathbf{D^n_2}\to 0$.

        For the term $\mathbf{D^n_3}$ we use Corollary~\ref{corollary: BDG} for $p=1$ and find
            \[ \mathbf{D^n_3} \le \mathcal{C}_p \mathbb{E}\left( (\sup_{0\le s\le t}\|X(s^-)\|^{p-1}) [M^n-M](t)^\frac{1}{2}\right) \]
        By Holder inequality we have
            \[ \le \mathcal{C}_p \mathbb{E} \left( \sup_{0\le s\le t}\|X(s^-)\|^p \right)^{\frac{p-1}{p}} \left(\mathbb{E} ([M^n-M](t)^\frac{p}{2})\right)^{\frac{1}{p}} \]
        where tends to 0 by Step 2. Hence $\mathbf{C^n_3}\to 0$.

      \item[(Step 7)] We claim that $\mathbb{E}|\mathbf{E^n}-\mathbf{E}| \to 0$. We have
        \[\begin{array} {l}
            \mathbb{E}|\mathbf{E^n}-\mathbf{E}| \le \underbrace{\mathbb{E}|\int_0^t (\|X^n(s^-)\|^{p-2}-\|X(s^-)\|^{p-2}) d[M^n]^c(s)|}_\mathbf{E^n_1}\\
            + \underbrace{\mathbb{E}|\int_0^t \|X(s^-)\|^{p-2} d([M^n]^c(s)-[M]^c(s))|}_\mathbf{E^n_2}.
        \end{array} \]

        For the term $\mathbf{E^n_1}$ we have
            \[ \mathbf{E^n_1} \le \mathbb{E}\left( (\sup_{0\le s\le t}|\|X^n(s^-)\|^{p-2}-\|X(s^-)\|^{p-2}|) [M^n]^c(t)\right) \]
        Now using the simple inequality $ |a-b|^r \le |a^r - b^r|$ for $r\ge 1$ and $a,b\in\mathbb{R}^+$ we have
        \[ |\|X^n(s^-)\|^{p-2} - \|X(s^-)\|^{p-2}| \le |\|X^n(s^-)\|^p-\|X(s^-)\|^p|^{\frac{p-2}{p}}. \]
        Substituting and using the Holder inequality we find
            \[ \le \left( \mathbb{E} \sup |\|X^n(s^-)\|^p -\|X(s^-)\|^p|\right)^\frac{p-2}{p} \left( \mathbb{E}[M^n]^c(t)^\frac{p}{2}\right)^\frac{2}{p} \]
        The second term above is bounded by $\left( \mathbb{E}[M](t)^\frac{p}{2}\right)^\frac{2}{p}$ since $[M^n]^c(t)\le[M](t)$. The first term, by the same arguments as in Step 5, after choosing a subsequence, converges to zero.

        For the term $\mathbf{E^n_2}$ we have
        \[ \mathbf{E^n_2} \le \mathbb{E}\left( (\sup_{0\le s\le t}\|X(s^-)\|^{p-2})([M]^c(t)-[M^n]^c(t))\right)\]
        By Holder inequality we have
        \[ \le \left( \mathbb{E} \sup_{0\le s\le t}\|X(s^-)\|^p\right)^\frac{p-2}{p} \left( \mathbb{E}  ([M]^c(t)-[M^n]^c(t))^\frac{p}{2}\right)^\frac{2}{p}. \]
        The first term is a constant. for the second term we have $0\le [M]^c(t)-[M^n]^c(t)\le [M](t)-[M^n](t) \le [M](t)$ and hence $([M]^c(t)-[M^n]^c(t))^\frac{p}{2}$ is dominated by $[M](t)^\frac{p}{2}$. on the other hand $\mathbb{E} ([M]^c(t)-[M^n]^c(t)) \le \mathbb{E} (\|M\|(t)^2 - \|M^n\|(t)^2) \to 0$. Hence $[M]^c(t)-[M^n]^c(t)$ and consequently $([M]^c(t)-[M^n]^c(t))^\frac{p}{2}$ tends to 0 in probability and therefore by Lebesgue's dominated convergence theorem it's expectation also tends to 0. Hence $\mathbf{E^n_2}\to 0$.

      \item[(Step 8)] We claim that $\mathbf{F^n}\to \mathbf{F}\,a.s$.
        We use the following lemma that is proved later.
        \begin{lemma}\label{lemma: inequality in proof of Ito-type pth power}
            For $x,y\in H$ we have
                \[ \|x+y\|^p-\|x\|^p -p\|x\|^{p-2}\langle x , y \rangle \le \frac{1}{2}p(p-1) (\|x\|^{p-2} +\|x+y\|^{p-2})\|y\|^2 \]
        \end{lemma}

        Note that the semimartingale $Z(s)$ is cadlag and hence is continuous except at a countable set of points $0\le s\le t$, and these are the only points in which the terms in the sums $\mathbf{F}$ and $\mathbf{F}^n$ are nonzero.

        By Lemma~\ref{lemma: inequality in proof of Ito-type pth power},
            \begin{equation}\label{equation: proof of Ito type pth power 4} \begin{array}{l} \left| \|X^n(s)\|^p - X^n(s^-)^p - p \|X^n(s^-)\|^{p-2} \langle X^n(s^-) , \Delta Z^n(s) \rangle \right|\\
            \le \frac{1}{2}p(p-1) (\|X^n(s^-)\|^{p-2} +\|X^n(s)\|^{p-2})\|\Delta Z^n(s)\|^2 \\
            \le p(p-1) (\sup_{0\le s\le t}\|X^n(s)\|^{p-2})\|\Delta Z^n(s)\|^2.
            \end{array} \end{equation}

        As in Step 5 we choose a subsequence $n_k$ for which there exists $\Omega_0\subset\Omega$ with $\mathbb{P}(\Omega_0)=1$ such that $\sup_{0\le s \le t} \| X^{n_k}(s)-X(s)\|^p \to 0, \textrm{ for }\omega\in\Omega_0$. Hence for $\omega\in\Omega_0$, $\|X^n(s)\| \to \|X(s)\|$ and in particular $\sup_n \sup_s \|X^n(s)\|^{p-2} <\infty$. Note also that $\|\Delta Z^n(s)\|^2\le \|\Delta Z(s)\|^2$ and that $\sum \|\Delta Z(s)\|^2 < \infty$. Hence by~\eqref{equation: proof of Ito type pth power 4}, for $\omega\in\Omega_0$, $\mathbf{F^n}$ is dominated by an absolutely convergent series. On the other hand since for $\omega\in\Omega_0$, $\|X^n(s)\| \to \|X(s)\|$ hence the terms of $\mathbf{F}^n$ converge to terms of $\mathbf{F}$. Hence by the dominated convergence theorem for series, we have $\mathbf{F^n}\to\mathbf{F}$ for $\omega\in\Omega_0$.
    \end{description}
\end{proof}

\begin{proof}[Proof of lemma~\ref{lemma: inequality in proof of Ito-type pth power}]
    Define $f(t)= \|x+ty\|^p$. Then
        \[ f'(t)= p\|x+ty\|^{p-2}\langle x+ty , y \rangle \]
    and
        \[ f''(t)= p\|x+ty\|^{p-2}\|y\|^2 +p(p-2)\|x+ty\|^{p-4} \langle x+ty , y \rangle ^2 \le p(p-1)\|x+ty\|^{p-2}\|y\|^2 \]
    By Taylor's remainder theorem we have for some $\tau\in [0,1]$,
        \[ f(1)-f(0)-f'(0)=\frac{1}{2}f''(\tau) \le \frac{1}{2} p(p-1)\|x+\tau y\|^{p-2}\|y\|^2 \]
    But $\|x+\tau y\|\le \max(\|x\|,\|x+y\|)$. Hence
        \[ f(1)-f(0)-f'(0) \le \frac{1}{2}p(p-1) (\|x\|^{p-2} +\|x+y\|^{p-2})\|y\|^2 \]
    which completes the proof.
\end{proof}

\section{Semilinear Stochastic Evolution Equations with L\'evy Noise and Monotone Nonlinear Drift}\label{section: Semilinear Stochastic Evolution Equations with Levy Noise and Monotone Nonlinearity}

In this section we will apply the theory developed in the last section to stochastic evolution equations. The noise term comes from a general L\'evy process and has Lipschitz coefficients, but the drift term is a non linear monotone operator.

The existence and uniqueness of the mild solutions of these equations in $L^2$ has been proved in~\cite{Salavati_Zangeneh:existence and uniqueness}. In this section we prove the existence and uniqueness of the solution in $L^p$ for $p\ge 2$ in Theorem~\ref{theorem: existence in L^p}. We also provide sufficient conditions under which the solutions are exponentially asymptotically stable.

Let $(\Omega,\mathcal{F},\mathcal{F}_t,\mathbb{P})$ be a filtered probability space. Let $(E,\mathcal{E})$ be a measurable space and $N(dt,d\xi)$ a Poisson random measure on $\mathbb{R}^+ \times E$ with intensity measure $dt \nu(d\xi)$. Our goal is to study the following equation in $H$,
\begin{equation}\label{main_equation}
    dX_t=AX_t dt+f(t,X_t) dt + \int_E k(t,\xi,X_{t-}) \tilde{N}(dt,d\xi),
\end{equation}
where $\tilde{N}(dt,d\xi)=N(dt,d\xi)-dt\nu(d\xi)$ is the compensated Poisson random measure corresponding to $N$.

We will use the notion of stochastic integration with respect to compensated Poisson random measure. For the definition and properties see~\cite{Peszat-Zabczyk} and~\cite{Albeverio-Mandrekar-Rudiger-2009}.

\begin{definition}
    $f:H\to H$ is called \emph{demicontinuous} if whenever $x_n \to x$, strongly in $H$ then $f(x_n)\rightharpoonup f(x)$ weakly in $H$.
\end{definition}

We assume the following,

\begin{hypothesis}\label{main_hypothesis}
    \begin{description}

        \item[(a)] $f(t,x,\omega):\mathbb{R}^+\times H\times \Omega \to H$ is measurable, $\mathcal{F}_t$-adapted, demicontinuous with respect to $x$ and there exists a constant $M$ such that
            \[ \langle f(t,x,\omega)-f(t,y,\omega),x-y \rangle \le M \|x-y\|^2,\]

        \item[(b)] $k(t,\xi,x,\omega):\mathbb{R}^+\times E\times H\times \Omega \to H$ is predictable and there exists a constant $C$ such that
            \[ \int_{E}\|k(t,\xi,x)-k(t,\xi,y)\|^2 \nu(d\xi) \le C \|x-y \|^2,\]

        \item[(c)] There exists a constant $D$ such that
            \[ \| f(t,x,\omega)\|^2 + \int_{E}\|k(t,\xi,x)\|^2 \nu(d\xi) \le D(1+\|x\|^2),\]
		
		\item[(d)] There exists a constant $F$ such that
            \[ \int_{E}\|k(t,\xi,x)-k(t,\xi,y)\|^p \nu(d\xi) \le F \|x-y \|^p,\]
            \[\int_{E}\|k(t,\xi,x)\|^p \nu(d\xi) \le F(1+\|x\|^p),\]
            
        \item[(e)] $X_0(\omega)$ is $\mathcal{F}_0$ measurable and $\mathbb{E}\|X_0\|^p < \infty$.
    \end{description}

\end{hypothesis}

\begin{definition}
    By a \emph{mild solution} of equation~\eqref{main_equation} with initial condition $X_0$ we mean an adapted c\`adl\`ag process $X_t$ that satisfies
    \begin{multline}\label{mild_solution}
        X_t=S_t X_0+\int_0^t S_{t-s}f(s,X_s) ds + \int_0^t{\int_E {S_{t-s}k(s,\xi,X_{s-})} \tilde{N}(ds,d\xi).}
    \end{multline}
\end{definition}

We will need an estimate for the $L^p$ norm of stochastic integrals with respect to compensated Poisson random measures. For this reason we state and prove the following theorem which is a Bichteler-Jacod inequality for Poisson integrals in infinite dimensions. This theorem is essentially the Lemma 4 of~\cite{Marinelli-Rockner-wellposedness} with an extension to $1\le p\le 2$. We provide a new proof for this theorem based on the Burkholder-Davies -Gundy inequality.

\begin{theorem}[An $L^p$ bound for Stochastic Integrals with Respect to Compensated Poisson Random Measures]\label{theorem: L^p_bound_stochastic_integral}
    Let $p\ge 1$. There exists a real constant denoted by $\mathfrak{C}_p$ such that if $k(t,\xi,\omega)$ is an $H$-valued predictable process for which the right hand side of \eqref{theorem: maximal inequality Poisson} is finite then
    \begin{equation} \label{theorem: maximal inequality Poisson}
    \begin{array}{l}
        \mathbb{E} \sup_{0\le t\le T} \left|\int_0^t\int_E k(s,\xi,\omega) \tilde{N}(ds,d\xi)\right|^p \le \\
        \mathfrak{C}_p \left( \mathbb{E} \left( \left(\int_0^T \int_E \left|k(s,\xi,\omega)\right| \nu(d\xi)ds\right)^{p} \right) + \mathbb{E} \int_0^T \int_E \left|k(s,\xi,\omega)\right|^p \nu(d\xi)ds\right)
    \end{array}
    \end{equation}
\end{theorem}

\begin{proof}
    Assume that $2^n\le p < 2^{n+1}$. We prove by induction on $n$.

    \emph{Basis of induction:} $n=0$. In this case we have $1\le p<2$ and the statement follows from Theorem 8.23 of~\cite{Peszat-Zabczyk}. In fact, in this case we have 
    \[\mathbb{E} \left|\int_0^t\int_E k(s,\xi,\omega) \tilde{N}(ds,d\xi)\right|^p \le 
        \mathfrak{C}_p \mathbb{E} \int_0^t \int_E \left|k(s,\xi,\omega)\right|^p \nu(d\xi)ds\]
    \emph{Induction Step:} Now assume $n\ge 1$ and we have proved the statement for $n-1$. Hence $p\ge 2$. Applying Burkholder-Davies-Gundy inequality we find
    \[ \mathbb{E} \sup_{0\le t\le T} \left|\int_0^t\int_E k(s,\xi,\omega) \tilde{N}(ds,d\xi)\right|^p \le K_p \mathbb{E} \left( (\int_0^T \int_E \|k(s,\xi,\omega)\|^2 N(ds,d\xi))^\frac{p}{2}\right) \]
	Subtracting a compensator from the right hand side we get
    \[\begin{array}{l}
        \le K_p 2^{\frac{p}{2}} \left( \mathbb{E} \left( \left(\int_0^T \int_E \left|k(s,\xi,\omega)\right|^2 \tilde{N}(ds,d\xi) \right)^\frac{p}{2}\right) + \mathbb{E} \left( \left(\int_0^T \int_E \left|k(s,\xi,\omega)\right|^2 \nu(d\xi)ds \right)^\frac{p}{2}\right) \right)
    \end{array}\]

    Note that $2^{n-1}\le \frac{p}{2} <2^n$ hence we can apply the induction hypothesis to the first term on the right hand side and find
    \begin{equation}\label{equation: proof of L^p_bound_1}
        \begin{array} {l}
        \le K_p 2^{\frac{p}{2}} \mathfrak{C}_{\frac{p}{2}}\left( \mathbb{E} \left( \left( \int_0^T \int_E \left|k(s,\xi,\omega)\right|^2 \nu(d\xi) ds \right)^{\frac{p}{2}}\right) + \mathbb{E} \left( \int_0^T \int_E \left|k(s,\xi,\omega)\right|^p \nu(d\xi) ds \right)\right) \\
        + K_p 2^{\frac{p}{2}} \mathbb{E} \left( \left(\int_0^T \int_E \left|k(s,\xi,\omega)\right|^2 \nu(d\xi)ds \right)^\frac{p}{2}\right)
        \end{array}
    \end{equation}
    By the interpolation inequality for a suitbale $\theta$ such that $\theta+\frac{1-\theta}{p}=\frac{1}{2}$ we have
    \[ \left( \int_0^T \int_E \left|k(s,\xi,\omega)\right|^2 \nu(d\xi) ds \right)^{\frac{1}{2}} \le \left( \int_0^T \int_E \left|k(s,\xi,\omega)\right| \nu(d\xi) ds \right)^{\theta} \left( \int_0^T \int_E \left|k(s,\xi,\omega)\right|^p \nu(d\xi) ds \right)^{\frac{1-\theta}{p}} \]
    raising to power $p$ we have
    \[ \left( \int_0^T \int_E \left|k(s,\xi,\omega)\right|^2 \nu(d\xi) ds \right)^{\frac{p}{2}} \le \left( \int_0^T \int_E \left|k(s,\xi,\omega)\right| \nu(d\xi) ds \right)^{\theta p} \left( \int_0^T \int_E \left|k(s,\xi,\omega)\right|^p \nu(d\xi) ds \right)^{1-\theta} \]
    By the arithmetic-geometric mean inequality
    \[ \le \theta \left( \int_0^T \int_E \left|k(s,\xi,\omega)\right| \nu(d\xi) ds \right)^{{p}} + (1-\theta)\left( \int_0^T \int_E \left|k(s,\xi,\omega)\right|^p \nu(d\xi) ds \right) \]
    taking expectations and substituting in~\eqref{equation: proof of L^p_bound_1} the statement is proved.
\end{proof}

Let $(\Omega,\mathcal{F},\mathcal{F}_t,\mathbb{P})$ be a filtered probability space and assume $f$ satisfies Hypothesis~\ref{main_hypothesis}-(a) and there exists a constant $D$ such that $\|f(t,x,\omega)\|^2 \le D (1+\|x\|^2)$ and assume $V(t,\omega)$ is an adapted process with c\`adl\`ag trajectories and $X_0(\omega)$ is $\mathcal{F}_0$ measurable.

We will need the following theorem,

\begin{theorem}[Zangeneh,~\cite{Zangeneh-measurability} and~\cite{Zangeneh-Thesis}]\label{theorem:measurability}
    With assumptions made above, the equation
        \[ X_t=S_t X_0 + \int_0^t S_{t-s}f(s,X_s,\omega) ds + V(t,\omega)\]
    has a unique measurable adapted c\`adl\`ag solution $X_t(\omega)$. Furtheremore
    \[	\|X(t)\|\le \|X_0\|+\|V(t)\|+\int_0^t e^{(\alpha+M)(t-s)} \|f(s,S_s X_0+V(s))\| ds, \]
\end{theorem}

The main theorem of this section,

\begin{theorem}[Existence of the Solution in $L^p$]\label{theorem: existence in L^p}
Let $p\ge 2$. Then under assumptions of Hypothesis~\ref{main_hypothesis}, equation~\eqref{main_equation} has a unique square integrable c\`adl\`ag mild solution $X(t)$ such that $\mathbb{E} \sup_{0\le s\le t} \|X(s)\|^p <\infty$.
\end{theorem}

\begin{lemma}\label{lemma: alpha=0}
        It suffices to prove theorem~\ref{theorem: existence in L^p} for the case that $\alpha=0$.
    \end{lemma}
    \begin{proof} Define
        \begin{gather*}
           \tilde{S}_t= e^{-\alpha t} S_t ,\qquad \tilde{f}(t,x,\omega)=e^{-\alpha t}f(t,e^{\alpha t}x,\omega) , \\
           \tilde{k}(t,\xi,x,\omega)=e^{-\alpha t}k(t,\xi,e^{\alpha t}x,\omega).
        \end{gather*}
        Note that $\tilde{S}_t$ is a contraction semigroup. It is easy to see that $X_t$ is a mild solution of equation~\eqref{main_equation} if and only if $\tilde{X}_t=e^{-\alpha t} X_t$ is a mild solution of equation with coefficients $\tilde{S},\tilde{f},\tilde{k}$.
    \end{proof}

\begin{proof}[Proof of Theorem~\ref{theorem: existence in L^p}.]
	Existence and uniqueness of the mild solution in $L^2$ has been proved in~\cite{Salavati_Zangeneh:existence and uniqueness}, Theorem 4. Uniqueness in $L^2$ implies the uniqueness in $L^p$ for $p\ge 2$. It remains to prove the existence in $L^p$.

    \emph{Existence.}
	It suffices to prove the existence of a solution on a finite interval $[0,T]$. Then one can show easily that these solutions are consistent and give a global solution. We define adapted c\`adl\`ag processes $X^{n}_t$ recursively as follows. Let $X^0_t=S_t X_0$. Assume $X^{n-1}_t$ is defined. Theorem~\ref{theorem:measurability} implies that there exists an adapted c\`adl\`ag solution $X^n_t$ of
    \begin{equation}\label{equation: proof of existence_iteration}
        X^n_t=S_t X_0 + \int_0^t S_{t-s}f(s,X^n_s) ds + V^n_t,
    \end{equation}
    where
        \[ V^{n}_t= \int_0^t{\int_E {S_{t-s}k(s,\xi,X^{n-1}_{s-})} \tilde{N}(ds,d\xi)}. \]
    
	It is proved in~\cite{Salavati_Zangeneh:existence and uniqueness} that $\{X^n\}$ converge to some adapted c\`adl\`ag process $X_t$ in the sense that 
	\[ \mathbb{E} \sup\limits_{0\le t\le T} \|X^n_t-X_t\|^2 \to 0, \]
	and that $X_t$ is the mild solution of equation~\eqref{main_equation}.
	
    We wish to show that $\{X^n\}$ converge to $X_t$ in $L^p$ with the supremum norm. This is done by the following two lemmas.
    
    \begin{lemma}\label{lemma: finite_pth_moment_iteration}
    	\[ \mathbb{E}\sup\limits_{0\le t\le T} \|X^n_t\|^p<\infty. \]
    \end{lemma}

	\begin{proof}
		We prove by induction on $n$. By Theorem~\ref{theorem:measurability} we have the following estimate,
	    \[	\|X^n_t\|\le \|X_0\|+\|V^n_t\|+\int_0^t e^{M(t-s)} \|f(s,S_s X_0+V^n_s)\| ds. \]
	  	Hence,
	   \[ \|X^n_t\|^p \le 3^p \|X_0\|^p + 3^p \|V^n_t\|^p + 3^p \left( \int_0^t e^{M(t-s)} \|f(s , S_s X_0+V^n_s)\| ds \right)^p \]
	  	Taking supremum and using Cauchy-Schwartz inequality we find	
	  \begin{multline*}
	  		\sup\limits_{0\le t\le T} \|X^n_t\|^p \le 3^p \|X_0\|^p + 3^p \sup\limits_{0\le t\le T} \|V^n_t\|^p\\
	  		+ 3^p e^{|M|T} T^{\frac{p}{2}} \left( \underbrace{\int_0^T \|f(s , S_s X_0+V^n_s)\|^2 ds}_G \right)^\frac{p}{2}
	  \end{multline*}
	  	Using Hypothesis~\ref{main_hypothesis}-(c) and Holder's inequality we find
	  	\begin{multline*}
	   		G \le D^\frac{p}{2} \left( \int_0^T (1+\|S_s X_0 +V^n_s\|^2) ds \right)^\frac{p}{2} \\
	   			\le D^\frac{p}{2} \left( T + 2T \|X_0\|^2 + 2 \int_0^T \|V^n_s\|^2 ds \right) ^\frac{p}{2} \\
	   			\le D^\frac{p}{2} \left( 3^\frac{p}{2} T^\frac{p}{2} + 2^\frac{p}{2}3^\frac{p}{2} T^\frac{p}{2} \|X_0\|^p + 2^\frac{p}{2}3^\frac{p}{2} T^\frac{p}{2} \sup\limits_{0\le s\le T} \|V^n_s\|^2 \right)
	   	\end{multline*}
	   	
	   	Hence, to prove the Lemma it suffices to prove that 
	    \[ \mathbb{E}\sup\limits_{0\le t\le T} \|V^n_t\|^p<\infty. \]

	    Applying Burkholder type inequality (Theorem~\ref{theorem:Burkholder type inequality}), we find
	    \[ \mathbb{E}\sup\limits_{0\le t\le T} \|V_t\|^p \le K_p \mathbb{E} ([\tilde{M}]_T^{\frac{p}{2}}), \]
	    where $\tilde{M}_t=\int_0^t \int_E k(s,\xi,X^{n-1}_{s-}) \tilde{N}(ds,du)$. Hence
	    \begin{multline*}
	    		\mathbb{E}\sup\limits_{0\le t\le T} \|V_t\|^p \le K_p \mathbb{E} \left( (\int_0^T \int_E \|k(s,\xi,X^{n-1}_{s})\|^2 N(ds,du))^\frac{p}{2}\right) \\
	    		\le 2^\frac{p}{2} K_p \Bigg( \mathbb{E} \left( (\int_0^T \int_E \|k(s,\xi,X^{n-1}_{s})\|^2 \nu(du) ds)^\frac{p}{2} \right) \\
	    	+ \mathbb{E} \left( (\int_0^T \int_E \|k(s,\xi,X^{n-1}_{s})\|^2 \tilde{N}(ds,du))^\frac{p}{2}\right) \Bigg)
	    \end{multline*}
	    By Hypothesis~\ref{main_hypothesis} (c) we have,
	    \begin{multline*}
	    		\le 2^\frac{p}{2} K_p D^\frac{p}{2} \left(\mathbb{E}( \int_0^T (1+\|X^{n-1}_s\|^2) ds)^\frac{p}{2})\right) \\
	    		+ 2^\frac{p}{2} K_p \mathbb{E} \left( (\int_0^T \int_E \|k(s,\xi,X^{n-1}_{s})\|^2 \tilde{N}(ds,du))^\frac{p}{2}\right)
	    \end{multline*}
	    Since $\frac{p}{2}\ge 1$, we can apply Theorem~\ref{theorem: L^p_bound_stochastic_integral} to second term and find
	    \begin{multline}\label{equation: proof of finite pth moment-1}
	    		\le 2^\frac{p}{2} K_p D^\frac{p}{2} \left(\mathbb{E}( \int_0^T (1+\|X^{n-1}_s\|^2) ds)^\frac{p}{2})\right) \\
	    		+ 2^\frac{p}{2} K_p \mathfrak{C}_p \left(\mathbb{E}( \int_0^T (1+\|X^{n-1}_s\|^2) ds)^\frac{p}{2})\right) \\
	    		+ 2^\frac{p}{2} K_p \mathfrak{C}_p \mathbb{E} \int_0^T (1+\|X^{n-1}_s\|^p) ds) \\
	    \end{multline}
	    Combining \eqref{equation:proof_of_finite_pth_moment-0} and \eqref{equation: proof of finite pth moment-1} we find by Hypothesis~\ref{main_hypothesis} (c) we have,
	    \begin{multline}\label{equation: proof of finite pth moment-4}
	    		\le 2^\frac{p}{2} \left((D\int_0^T \mathbb{E}\|X^{n-1}_s\|^2 ds)^\frac{p}{2}\right. \\
	    		\left. + \mathfrak{C}_\frac{p}{2} \left( \left( D\int_0^T \mathbb{E}\|X^{n-1}_s\|^2 ds \right)^\frac{p}{2} + D \left( \int_0^T \mathbb{E}\|X^{n-1}_s\|^p ds \right) \right) \right)\\
	    		  \le C_1 \left( (\int_0^T \mathbb{E}\|X^{n-1}_s\|^2 ds)^\frac{p}{2} \right) + C_2 \left( \int_0^T \mathbb{E}\|X^{n-1}_s\|^p ds \right) 
	    \end{multline}
	    where $C_1=2^\frac{p}{2} D (1+\mathfrak{C}_\frac{p}{2})$ and $C_2=2^\frac{p}{2} \mathfrak{C}_\frac{p}{2} D$, now by Holder inequality we find,
	    \begin{multline}\label{equation: proof of finite pth moment-5}
	    		\le C_3 \left( \int_0^T \mathbb{E}\|X^{n-1}_s\|^p ds \right) 
	    \end{multline}
	    which is finite by induction. The basis of induction follows directly from Hypothesis~\ref{main_hypothesis}-(e).
	\end{proof}

    \begin{lemma}\label{lemma: convergence_pth_moment}
        For $0\le t\le T$,
        \begin{equation}\label{equation: proof of existence_supremum convergence}
            \mathbb{E} \|X^{n+1}_t-X^n_t\|^p \le C_0 C_1^n \frac{t^n}{n!}
        \end{equation}
        where $C_0$ and $C_1$ are constants that are introduced below.
    \end{lemma}

    \begin{proof}
        We prove by induction on $n$. Assume that the statement is proved for $n-1$. We have,
        \begin{equation}\label{equation: proof of existence_X^n+1-X^n}
            X^{n+1}_t-X^n_t= \int_0^t S_{t-s}(f(s,X^{n+1}_s)-f(s,X^n_s)) ds + \int_0^t S_{t-s} dM_s,
        \end{equation}
        where
        \begin{eqnarray*}
            M_t &=&\int_0^t \int_E {(k(s,\xi,X^{n}_{s-})-k(s,\xi,X^{n-1}_{s-}))\tilde{N}(ds,d\xi)}.
        \end{eqnarray*}
        Applying Theorem~\ref{theorem: Ito type inequality for pth power}, for $\alpha=0$, we have
        \begin{multline}\label{equation:proof of existence 1}
            \lVert X^{n+1}_t-X^n_t \rVert ^p \le \\
            p \underbrace{\int_0^t {\|X^{n+1}_s-X^n_s\|^{p-2} \langle X^{n+1}_{s}-X^n_{s},f(s,X^{n+1}_s)-f(s,X^n_s)\rangle ds}}_{A^n_t}\\
            + p \underbrace{\int_0^t \|X^{n+1}_{s-}-X^n_{s-}\|^{p-2} \langle X^{n+1}_{s-}-X^n_{s-},dM_s\rangle}_{B^n_t}\\
		  + \frac{1}{2}p(p-1)  \underbrace{ \int_0^t \|X^{n+1}_{s-}-X^n_{s-}\|^{p-2} d[M]^c_s}_{C^n_t} + \int_0^t \int_E D^n_s N(ds,d\xi)
        \end{multline}
where
	\begin{multline*}
		D^n_s = \|X^{n+1}_{s-}-X^n_{s-} + k(s,\xi,X^n_{s-})-k(s,\xi,X^{n-1}_{s-})\|^p - \|X^{n+1}_{s-}-X^n_{s-}\|^p \\
		- p \|X^{n+1}_{s-}-X^n_{s-}\|^{p-2} \langle X^{n+1}_{s-}-X^n_{s-},k(s,\xi,X^n_{s-})-k(s,\xi,X^{n-1}_{s-})\rangle
	\end{multline*}
        Note that for a c\`adl\`ag function the set of discontinuity points is countable, hence when integrating with respect to Lebesgue measure, they can be neglected. Therefore from now on, we neglect the left limits in integrals with respect to Lebesgue measure. So, for the term $A_t$, the semimonotonicity assumption on $f$ implies
        \begin{equation}\label{equation:proof of existence_A_t}
            A^n_t \le M \int_0^t \|X^{n+1}_s-X^n_s\|^p ds
        \end{equation}
        We also have
        \[ [M]^c_t =0 \]
        and hence
        \begin{equation*}
           C^n_t =0
        \end{equation*}
		For the term $D^n_s$ we have by Lemma~\ref{lemma: inequality in proof of Ito-type pth power},
		\begin{multline*}
			D^n_s \le \frac{1}{2} p (p-1) \Bigg(  \|X^{n+1}_{s-}-X^n_{s-}\|^{p-2} \\
			+ \|k(s,\xi,X^n_{s-})-k(s,\xi,X^{n-1}_{s-})\|^{p-2} \Bigg) \|k(s,\xi,X^n_{s-})-k(s,\xi,X^{n-1}_{s-})\|^2
		\end{multline*}
		Hence by Hypothesis~\ref{main_hypothesis}-(b) and (d),
		\begin{multline}\label{equation:proof of existence_E(int D_s)}
        	\mathbb{E} \int_E D^n_s \nu(d\xi) \le \frac{1}{2} p (p-1) \Bigg(  C \mathbb{E} \left( \|X^{n+1}_{s-}-X^n_{s-}\|^{p-2}  \|X^{n}_s-X^{n-1}_s\|^2 \right) \\
        	+ F  \mathbb{E} \left( \|X^{n}_{s-}-X^{n-1}_{s-}\|^{p}\right) \Bigg)
        \end{multline}

        Now, taking expectations on both sides of~\eqref{equation:proof of existence 1} and substituting~\eqref{equation:proof of existence_A_t} and~\eqref{equation:proof of existence_E(C_t)} and~\eqref{equation:proof of existence_E(int D_s)} and noting that $B_t$ is a martingale we find,
        \begin{multline*}
        	\mathbb{E} \|X^{n+1}_t-X^n_t\|^p \le p M \int_0^t \mathbb{E}\|X^{n+1}_s-X^n_s\|^p ds \\
        	+ \frac{1}{2} p(p-1)C \int_0^t \mathbb{E}\left(\|X^{n+1}_s-X^n_s\|^{p-2}\|X^n_s-X^{n-1}_s\|^2\right)ds\\
        	+ \frac{1}{2} p (p-1) F \mathbb{E} \|X^n_s-X^{n-1}_s\|^p ds.
        \end{multline*}
        Applying Holder's inequality to the second integral in the right hand side we find
        \begin{multline*}
        	\le p M \int_0^t \mathbb{E}\|X^{n+1}_s-X^n_s\|^p ds \\
        	+ \frac{1}{2} p(p-1)C \left( \frac{p-2}{p}\int_0^t \mathbb{E}\|X^{n+1}_s-X^n_s\|^ ds + \frac{2}{p}\int_0^t \mathbb{E}\|X^n_s-X^{n-1}_s\|^p ds \right)\\
        	+ \frac{1}{2} p (p-1) F \mathbb{E} \|X^n_s-X^{n-1}_s\|^p ds\\
        	\le \beta  \int_0^t \mathbb{E}\|X^{n+1}_s-X^n_s\|^p ds + \gamma \int_0^t \mathbb{E}\|X^n_s-X^{n-1}_s\|^p ds
        \end{multline*}
        where $\beta= pM+ \frac{1}{2} (p-1)(p-2) C$ and $\gamma=\frac{1}{2} (p-1) (2C + pF)$.

	Define $h^n(t)=\mathbb{E}\|X^{n+1}_t-X^n_t\|^p$. We have
			\[ h^n(t) \le \beta \int_0^t h^n(s) ds + \gamma \int_0^t h^{n-1}(s) ds \]
       By Lemma~\ref{lemma: finite_pth_moment_iteration} we know that $h^n(t)$ is uniformly bounded with respect to $t$, hence we can use Gronwall's inequality and find
       	\begin{equation*}
            h^n(t)\le \gamma e^{\beta t} \int_0^t h^{n-1}(s) ds
        \end{equation*}
        We have $h^0(t)\le C_0$ where $C_0=2^p \mathbb{E} \sup_{0\le t \le T} (\|X^1_t\|^p + \|X^0_t\|^p)<\infty$ and it follows inductively that,
        	\[ h^n(t) \le C_0 C_1^n \frac{t^n}{n!} \]
        where $C_1=\gamma e^{\beta T}$.
       \end{proof}

    Back to the proof of Theorem~\ref{theorem: existence in L^p}, since the right hand side of~\eqref{equation: proof of existence_supremum convergence} is a convergent series, $\{X^n\}$ is a cauchy sequence in $L^p(\Omega,\mathcal{F},\mathbb{P};L^\infty([0,T];H))$ and hence converges to a process $Y_t(\omega)$. But as is proved in \cite{Salavati_Zangeneh:existence and uniqueness}, $\{X^n\}$ converges to a process $X_t$ in $L^2(\Omega,\mathcal{F},\mathbb{P};L^\infty([0,T];H))$ which is a solution of equation~\eqref{main_equation}. Hence $Y_t=X_t$.
\end{proof}

\begin{theorem} [Exponential Stability in the $p$th Moment]\label{theorem: stability}
    Let $X_t$ and $Y_t$ be mild solutions of~\eqref{main_equation} with initial conditions $X_0$ and $Y_0$. Then
    \begin{eqnarray*}
        \mathbb{E} \| X_t-Y_t \| ^p &\le&  e^{\gamma t} \mathbb{E}\|X_0-Y_0\|^p
    \end{eqnarray*}
    for $\gamma= p \alpha + p M+\frac{1}{2}p(p-1) C+ \frac{1}{2} p (p-1) ((2^{p-2}+1) C + 2^{p-2} F)$. In particular, if $\gamma < 0$ then all mild solutions are exponentially stable in the $L^p$ norm.
\end{theorem}

\begin{proof}
	    First we consider the case that $\alpha=0$. Subtract $X_t$ and $Y_t$,
    \begin{multline*}
        X_t-Y_t=S_t (X_0-Y_0)
        + \int_0^t S_{t-s} (f(s,X_s)-f(s,Y_s))ds
        + \int_0^t S_{t-s} dM_s,
    \end{multline*}
    where
    \[ M_t=\int_E (k(s,\xi,X_{s-})-k(s,\xi,Y_{s-}))d\tilde{N}. \]
    Applying It\"o type inequality (Theorem~\ref{theorem:ito type inequality}), for $\alpha=0$, to $X_t-Y_t$ and rewriting it with respect to random Poisson measure, we find
    \begin{multline}\label{equation:proof of stability 1}
        \| X_t-Y_t \| ^p \le \| X_0-Y_0 \| ^p \\
        +  p \underbrace {\int_0^t {\|X_s-Y_s\|^{p-2} \langle X_{s-}-Y_{s-} , (f(s,X_s)-f(s,Y_s))\rangle ds}}_\bold{A_t}\\
        + p \underbrace {\int_0^t {\|X_s-Y_s\|^{p-2} \langle X_{s-}-Y_{s-} , d M_s \rangle}}_\bold{B_t} \\
        +\frac{1}{2}p(p-1)\underbrace{\int_0^t \|X_s-Y_s\|^{p-2}d[M]_s}_\bold{C_t}        + \int_0^t \int_E  \bold{D_s} N(ds,d\xi)
    \end{multline}
where
    \begin{multline*}
        \bold{D_s}=\|X_{s-}-Y_{s-}+k(s,\xi,X_{s-})-k(s,\xi,Y_{s-})\|^p-\|X_{s-}-Y_{s-}\|^p \\
        -p\|X_{s-}-Y_{s-}\|^{p-2} \langle X_{s-}-Y_{s-},k(s,\xi,X_{s-})-k(s,\xi,Y_{s-}) \rangle .
    \end{multline*}
	
	Using Hypothesis~\ref{main_hypothesis} (a) for term $\bold{A}_t$ we find
	\begin{equation}\label{equation:proof of stability 2}
		\mathbb{E}\bold{A}_t \le M \int_0^t \mathbb{E} \|X_s-Y_s\|^p ds
	\end{equation}
	Using Hypothesis~\ref{main_hypothesis} (b) for term $\bold{C}_t$ we find
	\begin{equation}\label{equation:proof of stability 3}
		\mathbb{E}\bold{C}_t \le C \int_0^t \mathbb{E} \|X_s-Y_s\|^p ds
	\end{equation}
	
	For term $\bold{D}_s$ we have by Lemma~\ref{lemma: inequality in proof of Ito-type pth power},
	\begin{multline*}
		\bold{D}_s \le \frac{1}{2} p (p-1) \Big( \|X_{s-}-Y_{s-}\|^{p-2}
		+ \|X_{s-}-Y_{s-}+k(s,\xi,X_{s-})-k(s,\xi,Y_{s-})\|^{p-2}\Big)\\
		\|k(s,\xi,X_{s-})-k(s,\xi,Y_{s-})\|^2\\
		\le \frac{1}{2} p (p-1) \Big((2^{p-2}+1)\|X_{s-}-Y_{s-}\|^{p-2} + 2^{p-2} \|k(s,\xi,X_{s-})-k(s,\xi,Y_{s-})\|^{p-2}\Big)\\
		\|k(s,\xi,X_{s-})-k(s,\xi,Y_{s-})\|^2
	\end{multline*}
	Using Hypothesis~\ref{main_hypothesis} (b) and (d), we find
	\begin{multline}\label{equation:proof of stability 4}
		\mathbb{E} \int_E \bold{D}_s \nu(d\xi) ds \le \frac{1}{2} p (p-1) ((2^{p-2}+1) C + 2^{p-2} F)\mathbb{E} \|X_{s-}-Y_{s-}\|^p
	\end{multline}
    Taking expectations on both sides of~\eqref{equation:proof of stability 1} and noting that $\bold{B}_t$ is a martingale and substituting~\eqref{equation:proof of stability 2}, ~\eqref{equation:proof of stability 3} and ~\eqref{equation:proof of stability 4} we find
    
    \begin{equation*}
    	\mathbb{E} \|X_t-Y_t\|^p \le \mathbb{E} \|X_0-Y_0\|^p + \gamma \int_0^t \mathbb{E} \|X_s-Y_s\|^p ds \\
    \end{equation*}
    where $\gamma=p M+\frac{1}{2}p(p-1) C+ \frac{1}{2} p (p-1) ((2^{p-2}+1) C + 2^{p-2} F)$.
	Now applying Gronwall's inequality the statement follows. Hence the proof for the case $\alpha=0$ is complete. Now for the general case, apply the change of variables used in Lemma~\ref{lemma: alpha=0}.

\end{proof}

\begin{remark}
	 The results of this section remain valid by adding a Wiener noise term to equation~\eqref{main_equation}. i.e. for the equation
\begin{equation}\label{Wiener_noise}
    dX_t=AX_t dt+f(t,X_t) dt + g(t,X_{t-})d W_t + \int_E k(t,\xi,X_{t-}) \tilde{N}(dt,d\xi),
\end{equation}
	where $W_t$ is a cylindrical Wiener process on a Hilbert space $K$, independent of $N$ and $g(t,x,\omega):\mathbb{R}^+\times H\times \Omega \to L_{HS}(K,H)$ (Space of Hilbert-Schmidt operators from $K$ to $H$) is Lipschitz and has linear growth. The proofs are straightforward generalizations of the proofs of this section.
\end{remark}

\end{document}